\documentclass{birkmult}
 \newtheorem{thm}{Theorem}[section]

 \theoremstyle{definition}
 
 \theoremstyle{remark}

 \numberwithin{equation}{section}

\begin{document}
%
%
%
%
%
%
%
%
%
\title[Local moment problem]{Canonical solutions of the local moment problem}
\author[V. Adamyan]{Vadym Adamyan}

\address{%
Department of Theoretical Physics\\
Odessa National I.I. Mechnikov University\\
Dvoryanska 2\\
65044 Odessa\\
Ukraine}
\email{vadamyan@onu.edu.ua}

\thanks{This work was partly supported by the Erasmus Mundus grant EURO 1200752. Vadym Adamyan is grateful to the Universidad
Polit\'{e}cnica de Valencia for regular hospitality}
\author[I. Tkachenko]{Igor M. Tkachenko}
\address{Instituto de matem\'{a}tica Pura y Aplicada\\ Universidad
Polit\'{e}cnica de Valencia\\ Camino de Vera s/n\\ 46022 Valencia\\ Spain)}
\email{imtk@mat.upv.es}



\date{}
\dedicatory{}
\begin{abstract}
The present paper is devoted to the {\it local moment problem}, which consists in finding of non-decreasing functions on the real axis having given first $2n+1, \; n\geq 0,$ power moments on the whole axis and also $2m+1$ first power moments on a certain finite axis interval. Considering the local moment problem as a combination of the Hausdorff and Hamburger truncated moment problems we obtain the conditions of its solvability and describe the class of its solutions with minimal number of growth points if the problem is solvable.
 \end{abstract}
\maketitle
\section{Introduction}

 In many-body physical problems correlation functions of observables admit so far exact calculation only for infinitesimal time intervals and corresponding spectral distribution functions can be reconstructed from experimental data only for rather narrow spectral intervals. Attempts to extract from available data some useful information on correlation functions on the whole axis have led us to the following version of moment problem.

\textit{Given two sets of numbers (Hermitian matrices) }$a_{0},...,a_{2n}$, $%
b_{0},...,b_{2m}$ \textit{\ and an interval} $[0,\Lambda ],\,0<\Lambda
<\infty .$ \textit{To find a set of non-decreasing (matrix) measures }$%
d\sigma (t)$ \textit{, which satisfy the conditions:}

\begin{itemize}
\item
\begin{equation*}
\int\limits_{-\infty }^{\infty }t^{k}d\sigma(t)=a_{k}, \, k=0,...,2n;
\end{equation*}

\item
\begin{equation*}
\int\limits_{0}^{\Lambda }t^{l}d\sigma (t)=b_{l},\,l=0,...,2m.
\end{equation*}
\end{itemize}

The formulated problem is a special combination of the well known truncated Hausdorff and Hamburger
moment problems \cite{Akh, KN}. It is  motivated by the fact that in reality
the spectral distribution (matrix) function $\sigma (t)$ is accessible to observation
only on a finite spectral interval but at the same time the
values of a finite number of the first power moments of $d\sigma (t)$ can be
found independently from exact asymptotic relations and sum rules.

In this paper we study the local moment problem only for scalar $\sigma (t)$. Its first part contains a special class of solutions of the truncated Hausdorff moment problem. Using as in \cite{AT0, ATU, AT, AT2} the approach based on the extension theory of Hermitian operators we obtain in the next section the solvability criterium of the truncated Hausdorff problem, which is treated as a version of the Stieltjes problem \cite{Akh}, where the sought $\sigma (t)$ should be constant out of $(0, \Lambda )$.

In Section 3 we make clear here which among the canonical solutions of the truncated Stieltjes problem, that is solutions of minimal number of point mass, are constant out of $(0, \Lambda )$.

In a short Section 4 we present the Nevanlinna formula for description of the all canonical solutions of the truncated Hausdorff problem.

In the last section  a solution $\sigma(t)$ of the local moment moment problem ie represented as a sum $\sigma _{\Lambda }(t)+\sigma _{\Lambda }^{\perp }(t)$, where $\sigma _{\Lambda }(t)$ and $\sigma _{\Lambda }^{\perp }(t)$  grows only on $[0,\Lambda]$ and out of  $[0,\Lambda]$, respectively. The summand  $\sigma _{\Lambda }(t)$ on $[0,\Lambda]$ is nothing else but a solution of the Hausdorff problem for the given moments $b_{0},...,b_{2m}$, while $\sigma _{\Lambda }^{\perp }(t)$ is a solution of the truncated Hamburger moment for the altered moments $a_{0}-b_{0},...,a_{2n}-b_{2m}$, which has no growth points on $[0,\Lambda]$.  We find here for the latter problem, which we call the Hamburger problem with gap,  the solvability conditions and describe its canonical solutions.

\section{The solvability criterium of truncated Hausdorff moment problem}

The starting point for the solution of the above local moment problem is the
\textit{truncated Hausdorff moment problem}. It is
formulated as follows:

\textit{Given a set of real numbers}
\begin{equation}
\{b_{0},b_{1},b_{2},\ldots ,b_{2m}\},\quad m=0,1,2,\ldots .  \label{c}
\end{equation}%
\textit{To find all distributions }$\sigma (t)$\textit{\ such that}
\begin{equation}
\underset{0}{\overset{\Lambda }{\int }}t^{k}d\sigma (t)=b_{k}\ ,\
k=0,1,2,\ldots ,2m.  \label{e1}
\end{equation}

The formulation of the corresponding Stieltjes problem is similar, the only
difference is that $\Lambda $ in (\ref{e1}) is replaced by $\infty $,
\begin{equation}
\underset{0}{\overset{\infty }{\int }}t^{k}d\sigma (t)=b_{k}\ ,\
k=0,1,2,\ldots ,2m.  \label{e1a}
\end{equation}
Evidently, any solution of the Hausdorff problem is a special solution of
the Stieltjes problem, for which there are no growth points of the
 $\sigma (t)$ on the half-axis $(\Lambda ,\infty )$. Therefore the
criterium of solvability of the Stieltjes problem is only a necessary
condition for the solvability of the Hausdorff problem.

\begin{thm}\label{first}
A system of real numbers (\ref{c}) admits the representation (%
\ref{e1}) with non-decreasing $\sigma (t)$ if and only if

a) the Hankel matrix $\Gamma _{m}:=(b_{k+j})_{k,j=0}^{m}$ is non-negative;

b) for any set of complex numbers $\xi _{0},\ldots ,\xi _{r},\,0\leq r\leq
m-1,$ the condition
\begin{equation}
\underset{j,k=0}{\overset{r}{\sum }}b_{j+k}\xi _{k}\overline{\xi _{j}}=0
\label{add}
\end{equation}%
implies
\begin{equation}
\underset{j,k=0}{\overset{r}{\sum }}b_{j+k+2}\xi _{k}\overline{\xi _{j}}=0;
\label{add1a}
\end{equation}%
c) the Hankel matrix $\Gamma _{m-1}^{(1)}:=(b_{k+j+1})_{k,j=0}^{m-1}$ is
non-negative and for any set $\xi _{0},...,\xi _{r}\in\mathbb{C},\,\, 0\leq r\leq m-1,$
the condition
\begin{equation}
\underset{j,k=0}{\overset{r}{\sum }}b_{j+k+1}\xi _{k}\overline{\xi _{j}}=0
\label{add1b}
\end{equation}%
implies (\ref{add1a});

d) the matrix $\Lambda \Gamma _{m-1}-\Gamma _{m-1}^{(1)}$ is non-negative definite.

\end{thm}

\begin{proof}
Due to \cite{ATU, AT}) the conditions a) - c) of the theorem is a
criterion of solvability of the truncated Stieltjes moment problem.
Therefore we need only to prove that the condition d), in addition to a) -
c), is equivalent to the existence, for given moments, of those solutions of the Stieltjes
problem, for which $\sigma (t)=const$ for $t>\Lambda $.

Notice that due to the conditions a) and c) of the theorem, the moments $%
b_{j}$ are non-negative, $b_{j}\geq 0,\,j=0,...,2m$. Excluding the trivial
case, when the sought $\sigma \left( t\right) $ may have only one point of
growth at $t=0,$ from now on we will assume that all these numbers are
strictly positive, i.e. $b_{j}>0,$ $\,j=0,...,2m$.

Suppose that a) - d) hold. In this case for a given set of real
numbers $b_{0},...,b_{2m}$ by virtue of the conditions a)- c) the
corresponding truncated Stieltjes moment problem has at least one solution $\sigma_{0} (t)$
 which amounts to $s\leq m$ point masses $\mu_{1},...\mu_{s}, \, \min{\mu_{j}>0},$ located at some points $t_{1}<...<t_{s}<\infty$ of the half-axis $[0,\infty)$ \cite{ATU} (and also \cite{AT})(We will return to this issue later). Note that \textit{the distribution   $\sigma_{0} (t)$ is at the same time a solution of the Hausdorff problem if and only if} $t_{s}\leq \Lambda$.
For an arbitrary set of complex numbers $\xi _{0},\xi _{1},\ldots ,\xi _{m-1}$ and the polynomial
\begin{equation}
P(t)=\xi _{0}+\xi _{1}\ t+\xi _{2}\ t^{2}+\ldots +\xi _{r}\ t^{m-1}
\label{defa}
\end{equation}%
the equalities (\ref{e1}) for $\sigma_{0} (t)$ and the special form of this distribution result in equalities
\begin{equation}\label{e2}
\begin{array}{c}
\underset{0}{\overset{\infty }{\int }}(\Lambda -t)\left\vert
P(t)\right\vert ^{2}\ d\sigma \left( t\right)=\Lambda \cdot \underset{j,k=0}{\overset{m-1}{\sum }}b_{j+k}\xi _{k}\overline{\xi _{j}}-\underset{j,k=0}{\overset{m-1}{\sum }}b_{j+k+1}\xi _{k}\overline{\xi _{j}}= \\
=\underset{j=1}{\overset{s}{\sum }}\left(\Lambda-t_{j}\right)\left|P(t_{j}) \right|^{2}\mu_{j}.
\end{array}
\end{equation}%
By (\ref{e2}) if $t_{s}\leq \Lambda$ then, evidently the matrix $\Lambda \cdot \Gamma _{m-1}-\Gamma _{m-1}^{(1)}$ is
non-negative definite.

Alternatively, if $t_{s}>\Lambda$, then for
\[\begin{array}{c}
Q(t)=\eta_{0}+\eta_{1}t+...+\eta_{m-1}t^{m-1}, \\
Q(t_{1})=Q(t_{2})=...=Q(t_{s-1})=0, \,\, Q(t_{s})=1,
\end{array}
\]
we see that
\[
\Lambda \cdot \underset{j,k=0}{\overset{m-1}{\sum }}b_{j+k}\xi _{k}\overline{\eta _{j}}-\underset{j,k=0}{\overset{m-1}{\sum }}b_{j+k+1}\eta _{k}\overline{\eta _{j}}=\underset{j=1}{\overset{s}{\sum }}\left(\Lambda-t_{j}\right)\left|Q(t_{j}) \right|^{2}\mu_{j}=\left(\Lambda - t_{s}\right)\mu_{s}<0,
\]
what is incompatible with the condition d) of the theorem.
\end{proof}

\section{Which canonical solutions of the truncated Stieltjes problem are also  solutions of the truncated Hausdorff problem?}

Let us assume that for a system of real numbers (\ref{c}) the conditions a) - c) of Theorem \ref{first} hold and let $\sigma (t)$ be some solution of the corresponding truncated Stieltjes problem. Taking the set $\mathcal{L}$ of continuous complex valued functions $f\left( t\right) ,\;0<t<\infty ,$
for which
\begin{equation}
\int\limits_{0}^{\infty }\left\vert f(t)\right\vert ^{2}d\sigma \left(
t\right) <\infty,  \label{pre1}
\end{equation}%
we will consider $\mathcal{L}$ as pre-Hilbert space  with the bilinear functional
\begin{equation}
\left\langle f,g\right\rangle =\int\limits_{0}^{\infty }f(t)\overline{%
g\left( t\right) }\ d\sigma \left( t\right), \,\, f,g \in \mathcal{L}.  \label{pre2}
\end{equation}%
as the scalar product.
Due to the conditions (\ref{e1a}), any polynomial
\begin{equation}
f\left( t\right) =\xi _{0}+\xi _{1}t+\ldots +\xi _{r}t^{r},\;\xi _{0},\ldots
,\xi _{r}\in \mathbb{C}, \,\, 0\leq r \leq m,   \label{pre3}
\end{equation}%
may be considered as an element of $\mathcal{L}$. We will denote the linear
subset of such polynomials by $\mathcal{P}_{m}$.

Let $\mathcal{L}_{0}$ be the subspace of $\mathcal{L}$ consisting of all
functions $f$ such that $\left\Vert f\right\Vert :=\sqrt{\left\langle
f,f\right\rangle }=0$ and $\mathcal{\tilde{L}}$ be the factor - space $%
\mathcal{L}\backslash \mathcal{L}_{0}$. For any class of elements $\hat{g}=f+%
\mathcal{L}_{0}$ of this factor space we set $\left\Vert \hat{g}\right\Vert
_{\widetilde{\mathcal{L}}}=\left\Vert f\right\Vert $. Taking the closure of $%
\mathcal{\tilde{L}}$ with respect to this norm, we obtain the Hilbert space $%
\mathrm{L}_{\sigma }^{2}$. We keep the same symbol $\left\langle
.,.\right\rangle $ for the scalar product in $\mathrm{L}_{\sigma }^{2}$.

Let $\mathrm{L}_{m}$ be the subspace of $\mathrm{L}_{\sigma }^{2}$ generated by
the subset of polynomials $\mathcal{P}_{m}$. By (\ref{e1}) and (\ref{pre2})
for $f,g$ $\mathbf{\in }$ $\mathcal{P}_{m},$%
\begin{equation}
f\left( t\right) =\sum\limits_{r=0}^{m}\xi _{r}t^{r},\;g\left( t\right)
=\sum\limits_{r=0}^{m}\eta _{r}t^{r},\;\xi _{0},\ldots ,\eta _{r}\in \mathbb{%
C},
\end{equation}%
we have
\begin{equation}
\left\langle f,g\right\rangle =\underset{j,k=0}{\overset{m}{\sum }}%
b_{j+k}\xi _{k}\overline{\eta _{j}}.  \label{triv1}
\end{equation}%
Therefore for all distributions $\sigma \left( t\right) $ satisfying (\ref%
{e1}), the restrictions onto $\mathrm{L}_{m}$ of the scalar products
in the corresponding spaces $\mathrm{L}_{\sigma }^{2}$ must coincide. Among
non-decreasing functions $\sigma \left( t\right) $ satisfying (\ref{e1}),
those for which $\mathrm{L}_{\sigma }^{2}=\mathrm{L}_{m}$ are referred to as
\textit{canonical}. It was proven in \cite{ATU} that the set of canonical
solutions of the truncated Stieltjes moment problem is non-empty whenever
the latter is solvable, i.e. whenever the conditions a) - c) of the theorem
hold. By (\ref{triv1}), a canonical $\sigma \left( t\right) $ is a
non-decreasing function having only a finite number $\leq m$ of growth points.

Take some canonical solution $\tilde{\sigma}\left( t\right) $ of the
truncated Stieltjes moment problem for the given set of moments and consider
the self-adjoint operator $\widetilde{A}$ of multiplication by the
independent variable $t$ in the related space $\mathrm{L}_{\widetilde{\sigma
}}^{2}=\mathrm{L}_{m}$. Take the class $\widehat{e}_{0}\subset \mathrm{L}_{m}
$ containing the polynomial $\widehat{e}_{0}(t)\equiv 1$ and the classes
containing the polynomials $\widehat{e}_{k}(t)\equiv t^{k},\;0\leq k\leq m$.
According to the definition of $\widetilde{A}$ we have the representation
\begin{equation}
\widehat{e}_{k}=\widetilde{A}^{k}\widehat{e}_{0}\text{ , }0\leq k\leq m.
\label{defsga}
\end{equation}%
For the unity decomposition $\widetilde{E}_{t},$ $-\infty <t<\infty ,$ of $%
\widetilde{A},$ let us introduce a non - decreasing function $\widetilde{%
\sigma }\left( t\right) ,$ $-\infty <t<\infty $ of bounded variation
\begin{equation}
\widetilde{\sigma }\left( t\right) :=\left\langle \widetilde{E}_{t}\widehat{e%
}_{0},\widehat{e}_{0}\right\rangle _{\mathrm{L}_{m}}.  \label{defsig}
\end{equation}%
By (\ref{defsga}), (\ref{defsig}), and (\ref{e1})
\begin{equation}
b_{j+k}=\left\langle \widehat{e}_{k},\widehat{e}_{j}\right\rangle _{\mathrm{L%
}_{m}}=\left\langle \widetilde{A}^{k}\widehat{e}_{0},\widetilde{A}^{j}%
\widehat{e}_{0}\right\rangle _{\mathrm{L}_{m}}=
\end{equation}%
\begin{equation}
=\int\limits_{0}^{\infty }t^{j+k}d\left\langle \widetilde{E}_{t}\widehat{e}%
_{0},\widehat{e}_{0}\right\rangle _{\mathrm{L}_{n}}=\int\limits_{0}^{\infty
}t^{j+k}d\widetilde{\sigma }\left( t\right) ,\quad \,0\leq j,k\leq m.
\end{equation}%
Let us denote by $\mathrm{L}_{m-1}$ the subspace of $\mathrm{L}_{m}$
generated by polynomials of a degree $\leq m-1$. By definition of $\widetilde{A}$ its restriction $A_{0}$ to the subspace $\mathrm{L}_{m-1}$ is a symmetric operator which actually does not depend on the choice
of the canonical solution of the truncated Stieltjes moment problem.
Therefore each canonical solution $\widetilde{\sigma }\left( t\right) $ of
this problem generates some self-adjoint extension $\widetilde{A}$ of $A_{0}$
in $\mathrm{L}_{m}$. On the other hand, each canonical self-adjoint
extension $\widetilde{A}$ of $A_{0}$ in $\mathrm{L}_{m}$ generates a certain
solution $\widetilde{\sigma }\left( t\right) $ of the truncated Stieltjes
moment problem. By the above formulas such a solution is at the same time a
solution of the Hausdorff problem if and only if the corresponding spectral
function $\widetilde{E}_{t}$ has no points of growth on the half-axis $%
(\Lambda ,\infty )$, i.e. if and only if $\Lambda \cdot I_{m}-\widetilde{A}$%
, where $I_{r}$ is the unity operator in $\mathrm{L}_{r}$, is a non-negative
extension of $\Lambda \cdot I_{m-1}-A_{0}$. Such an extension of $\Lambda
\cdot I_{m-1}-A_{0}$ may exist only if the operator $\Lambda \cdot
I_{m-1}-A_{0}$ is itself non-negative, i.e. the quadratic form of $\Lambda
\cdot I_{m-1}-A_{0}$ is non-negative. But this is the case, since by our
assumptions for a class $\hat{f}\in \mathrm{L}_{m-1}$ containing a
polynomial
\begin{equation}
f\left( t\right) =\sum\limits_{r=0}^{m-1}\xi _{r}t^{r},
\end{equation}%
we have by (\ref{e1a})
\begin{equation}\begin{array}{c}
\begin{array}{cc}
\left\langle \hat{f},\hat{f}\right\rangle _{\mathrm{L}_{m-1}}=\underset{j,k=0}{\overset{m-1}{\sum }%
}b_{j+k}\xi _{k}\overline{\xi _{j}}, & \left\langle A_{0}\hat{f},\hat{f}%
\right\rangle _{\mathrm{L}_{m-1}}= \underset{j,k=0}{\overset{m-1}{\sum }%
}b_{j+k+1}\xi _{k}\overline{\xi _{j}}, \end{array} \\
\left\langle \lbrack \Lambda \cdot I_{m-1}-A_{0}]\hat{f},\hat{f}%
\right\rangle _{\mathrm{L}_{m-1}}=\Lambda \underset{j,k=0}{\overset{m-1}{\sum }%
}b_{j+k}\xi _{k}\overline{\xi _{j}}-\underset{j,k=0}{\overset{m-1}{\sum }}%
b_{j+k+1}\xi _{k}\overline{\xi _{j}}\geq 0. \end{array} \label{nonneg1}
\end{equation}%
If $\mathrm{L}_{m}=\mathrm{L}_{m-1}$, i.e. if $\det \Gamma _{m}=0$, then $%
A_{0}$ is a self-adjoint operator and in this case the truncated Stieltjes
problem has a unique solution $\sigma _{0}(t)$, which is, in line with (\ref{defsig}), generated
by the spectral function $E_{t}^{0}$ of $A_{0}$.
Since
\begin{equation*}
\Lambda \cdot I_{m-1}-A_{0}\geq 0,
\end{equation*}%
then $\sigma _{0}(t)$ is also the unique solution of the truncated Hausdorff
problem.

To describe the class of canonical solutions of Hausdorff
problem if $\mathrm{L}_{m}\neq \mathrm{L}_{m-1}$, i.e. if $\det \Gamma _{m}>0,
$ we remind first how it is done in the less restrictive case of Stieltjes problem.

Note that the condition $\det \Gamma _{m}>0$ according to which $\Gamma _{m}>0$ yields also $\Gamma _{m}^{(1)}>0$ . Indeed, if the quadratic form in of $\Gamma _{m}^{(1)}\geq 0$ vanishes for some set of complex numbers
\begin{equation*}
\xi _{0},...,\xi _{m-1},\underset{0\leq k\leq m-1}{\max }\left\vert \xi
_{k}\right\vert >0,
\end{equation*}%
then, by the condition c) of the theorem, the quadratic form of matrix $\Gamma
_{m-1}^{(2)}=(s_{j+k+2})_{j,k=0}^{m-1}$ also vanishes for the same set and hence $\Gamma _{m-1}^{(2)}$ is non-invertible. But $\Gamma_{m-1}^{(2)}$ is a diagonal block of positive definite matrix $\Gamma _{m}$,
a contradiction.

Let $\mathcal{N}=\mathrm{L}_{m}\ominus \mathrm{L}_{m-1}$, $\dim \mathcal{N}=1
$ and  $P_{\mathcal{N}}$ be the orthogonal projector onto
the one-dimensional subspace $\mathcal{N}$. With respect to the representation of $\mathrm{L}_{m}$ as the orthogonal
sum $\mathrm{L}_{m-1}\oplus \mathcal{N}$ , we can represent a self-adjoint
extension $\widetilde{A}$ of $A_{0}$ as a $2\times 2$ block operator matrix
\begin{equation}
\widetilde{A}=\left(
\begin{array}{cc}
A_{00} & G^{\ast } \\
G & \widetilde{H}%
\end{array}%
\right) ,  \label{block}
\end{equation}%
where $A_{00}$ is a symmetric operator in $\mathrm{L}_{m-1}$, the quadratic
form of which coincides with that of $A_{0}$, $G=P_{\mathcal{N}}A_{0\mid
\mathrm{L}_{m-1}}$  and $%
\widetilde{H}$ is a self-adjoint operator in $\mathcal{N}$, which just specifies
a certain extension $\widetilde{A}$. By (\ref{nonneg1}) $A_{00}$ is a positive
definite operator. Using the Schur-Frobenius factorization we can represent $%
\widetilde{A}$ in the form
\begin{equation}
\widetilde{A}=\left(
\begin{array}{cc}
I & 0 \\
GA_{00}^{-1} & I%
\end{array}%
\right) \left(
\begin{array}{cc}
A_{00} & 0 \\
0 & \widetilde{H}-GA_{00}^{-1}G^{\ast }%
\end{array}%
\right) \left(
\begin{array}{cc}
I & A_{00}^{-1}G^{\ast } \\
0 & I%
\end{array}%
\right) .  \label{Schur}
\end{equation}%
By this representation the extension $\widetilde{A}\geq 0$ if and only if $%
\widetilde{H}\geq GA_{00}^{-1}G^{\ast }$. We see that those and only those self-adjoint operators $\widetilde{H}$ in $\mathcal{N}$ which have form
\begin{equation}
\widetilde{H}=GA_{00}^{-1}G^{\ast }+Q,\,Q\geq 0,  \label{canonic}
\end{equation}%
with a non-negative operator $Q$ in $\mathcal{N}$, generate non-negative extensions $\widetilde{A}$ in $\mathrm{L}_{m}$ of $A_{0}$.  and thereby generate canonical solutions of the Stieltjes
problem. But only those of them are solutions of the Hausdorff problem, for
which the corresponding non-negative extension $\widetilde{A}$ satisfies the
condition
\begin{equation}
\widetilde{A}-\Lambda \cdot I_{m}\leq 0.  \label{hausd}
\end{equation}%
To express the condition (\ref{hausd}) in terms of given moments (\ref{c}) let us consider the Schur-Frobenius representation for $\widetilde{A}-\lambda \cdot I_{m}$ assuming that $\lambda >\Lambda $.
Due to the condition d), this guarantees the invertibility of $%
A_{00}-\lambda \cdot I_{m-1}$. We have
\begin{equation}
\begin{array}{c}
\widetilde{A}-\lambda \cdot I_{m}=\left(
\begin{array}{cc}
I & 0 \\
G[A_{00}-\lambda \cdot I_{m-1}]^{-1} & I%
\end{array}%
\right) \times  \\
\left(
\begin{array}{cc}
A_{00}-\lambda \cdot I_{m-1} & 0 \\
0 & \widetilde{H}-\lambda \cdot I_{\mathcal{N}}-G[A_{00}-\lambda \cdot
I_{m-1}]^{-1}G^{\ast }%
\end{array}%
\right) \times  \\
\left(
\begin{array}{cc}
I & [A_{00}-\lambda \cdot I_{m-1}]^{-1}G^{\ast } \\
0 & I%
\end{array}%
\right) .%
\end{array}
\label{Schur1}
\end{equation}%
By virtue of (\ref{Schur1}), an extension $\widetilde{A}$ satisfies the
condition
\begin{equation}
\widetilde{A}-\lambda \cdot I_{m}\leq 0  \label{ineq}
\end{equation}%
if and only if $A_{00}-\lambda \cdot I_{m-1}<0$, what is provided by the
condition d), and
\begin{equation}
\widetilde{H}-\lambda \cdot I_{\mathcal{N}}-G[A_{00}-\lambda \cdot
I_{m-1}]^{-1}G^{\ast }<0.  \label{hausd1}
\end{equation}%
Let us denote by $A_{\mu }$ the \textit{minimal} non-negative extension of $%
A_{0}$, for which $Q=0$ in (\ref{canonic}). This and only
this canonical extension is non-invertible. For $A_{\mu }$ the block $%
\widetilde{H}$ is simply $GA_{00}^{-1}G^{\ast }$. The inequality (\ref%
{hausd1}) holds for some non-negative extension $\widetilde{A}$ if and only
it is true for the minimal extension $\widetilde{A}_{\mu }$ in (\ref{canonic}%
), that is if
\begin{equation}
GA_{00}^{-1}G^{\ast }-\lambda I_{\mathcal{N}}-G[A_{00}-\lambda \cdot
I_{m-1}]^{-1}G^{\ast }\leq 0,\,\lambda \leq \Lambda .  \label{cond1}
\end{equation}%
Since the function of $\lambda $ in the left hand side of (\ref{cond1}) is
non-increasing, then the extension $\widetilde{A}_{0}$ satisfies the
inequality (\ref{ineq}) if and only if
\begin{equation}
\Lambda \geq G[\Lambda \cdot I_{m-1}-A_{00}]^{-1}G^{\ast
}+GA_{00}^{-1}G^{\ast }.  \label{cond2}
\end{equation}%
In what follows, $\left\{e_{k}\right\}_{0}^{m}$ denote the natural basis $\mathfrak{B}$ in $\mathrm{L}_{m}$ of monomials $\left\{t^{k}\right\}_{0}^{m}$. To represent $GA_{00}^{-1}G^{\ast }$ in a more explicit form we introduce
in $\mathrm{L}_{m}$ operators $P_{\mathfrak{N}}$ and $T$ in $\mathrm{L}_{m}$, which for the basis $\mathfrak{B}$ act as multiplication by $(m+1)\times (m+1)$  matrices
\begin{equation*}\begin{array}{cc}

P_{\mathfrak{N}}=\left(
\begin{array}{cccc}
0 & 0& \cdots  & 0 \\
0 & 0& \cdots  & 0 \\
\vdots  & \vdots &\vdots & \vdots  \\
0 & 0&\cdots  & 1%
\end{array}%
\right),  &
 T=\left(
\begin{array}{ccccc}
0 & 0 & 0 & \cdots  & 0 \\
1 & 0 & 0 & \cdots  & 0 \\
0 & 1 & 0 & \cdots  & 0 \\
\vdots  & \vdots  & \vdots  & \vdots  & \vdots  \\
0 & 0 & \cdots  & 1 & 0%
\end{array}%
\right) .
\end{array}
\end{equation*}%
The symmetric operator $A_{0}$ in $\mathrm{L}_{m}$ is the restriction of $T$
to the subspace $\mathrm{L}_{m-1}$. Let $\widetilde{\Gamma }_{m-1}^{(1)}$ be
the $(m+1)\times (m+1)$ block operator matrix
\begin{equation}
\widetilde{\Gamma }_{m-1}^{(1)}=\left(
\begin{array}{cc}
\Gamma _{m-1}^{(1)} & 0_{m,1} \\
0_{1,m} & 0_{1,1}%
\end{array}%
\right) ,
\end{equation}%
where $0_{n,m}$ are the $n\times m$ null-matrices. Note that for $\mathbf{%
\xi }\in \mathrm{L}_{m-1}$ and any $\mathbf{\eta }\in \mathrm{L}_{m}$ we
have
\begin{eqnarray*}
&<&A_{0}\mathbf{\xi ,\eta }>_{ \mathrm{L}_{m}}=<T\mathbf{\xi ,\eta }>_{ \mathrm{L}_{m}}=\left( \widetilde{\Gamma
}_{m-1}^{(1)}\mathbf{\xi ,\eta }\right)_{\mathbb{C}_{m}} +\left( P_{\mathfrak{N}}\Gamma _{m}T%
\mathbf{\xi ,\eta }\right)_{\mathbb{C}_{m}}  \\
&=&<\Gamma _{m}^{-1}\widetilde{\Gamma }_{m-1}^{(1)}\mathbf{\xi ,\eta >}%
+<\Gamma _{m}^{-1}P_{\mathfrak{N}}\Gamma _{m}T\mathbf{\xi ,\eta >.}
\end{eqnarray*}%
Hence
\begin{equation}
A_{0\mid \mathbb{C}_{m-1}}=\Gamma _{m}^{-1}\widetilde{\Gamma }_{m-1\mid
\mathbb{C}_{m-1}}^{(1)}+\Gamma _{m}^{-1}P_{\mathfrak{N}}\Gamma _{m}T_{\mid
\mathbb{C}_{m-1}}.  \label{azero}
\end{equation}%
By (\ref{azero}) any self-adjoint extension $\widetilde{A}$ of $A$ in $%
\mathrm{L}_{n}$ has the form
\begin{eqnarray}
\widetilde{A} &=&\Gamma _{m}^{-1}\widetilde{\Gamma }_{m-1}^{(1)}P_{\mathfrak{%
N}}^{\perp }+\Gamma _{m}^{-1}P_{\mathfrak{N}}\Gamma _{m}TP_{\mathfrak{N}%
}^{\perp }+  \label{aext} \\
&&\Gamma _{m}^{-1}P_{\mathfrak{N}}^{\perp }T^{\ast }\Gamma _{m}P_{\mathfrak{N%
}}+\Gamma _{m}^{-1}\widetilde{H},  \notag
\end{eqnarray}%
where $P_{\mathfrak{N}}^{\perp }=I-P_{\mathfrak{N}}$,
\begin{equation*}
\widetilde{H}=\left(
\begin{array}{cc}
0 & 0 \\
0 & H%
\end{array}%
\right) ,
\end{equation*}%
 and $H$ is some real number,
which defines the extension $\widetilde{A}$. In a more detailed form,
\begin{eqnarray}
\widetilde{A} &=&\Gamma _{m}^{-1}\left(
\begin{tabular}{ll}
$\quad \quad \ \ \Gamma _{m-1}^{(1)}$ & $%
\begin{array}{c}
b_{m+1} \\
\vdots  \\
b_{2m}%
\end{array}%
$ \\
$%
\begin{array}{ccc}
b_{m+1} & \cdots  & b_{2n}%
\end{array}%
$ & $\quad H$%
\end{tabular}%
\ \right) =  \label{aext1} \\
&=&T+\Gamma _{m}^{-1}\left(
\begin{tabular}{ll}
$\quad \quad 0_{m,m}$ & $%
\begin{array}{c}
b_{m+1} \\
\vdots  \\
b_{2m}%
\end{array}%
$ \\
$%
\begin{array}{ccc}
0 & \cdots  & 0%
\end{array}%
$ & $\quad H$%
\end{tabular}%
\ \right) .  \label{aext2}
\end{eqnarray}%
Observe, as before, that the invertibility of $\Gamma _{m}$ and the
condition c) of Theorem \ref{first} guarantee the invertibility of the matrix $%
\Gamma _{m-1}^{(1)}$. Write $\Gamma _{m-1}^{(1)-1}=(s_{jk})_{j,k=0}^{m-1}$
and put
\begin{equation*}
\left( \widetilde{\Gamma }_{m-1}^{(1)}\right) _{\mathrm{cond}}^{-1}=\left(
\begin{array}{cc}
\Gamma _{m-1}^{(1)-1} & 0_{m,1} \\
0_{1,m} & 0_{1,1}%
\end{array}%
\right)
\end{equation*}%
By the above argument the operator defined by the block matrix (\ref{aext1})
is non-negative if and only if
\begin{equation*}
\widetilde{H}-P_{\mathfrak{N}}\Gamma _{m}TP_{\mathfrak{N}}^{\perp }\left(
\widetilde{\Gamma }_{m-1}^{(1)}\right) _{\mathrm{cond}}^{-1}P_{\mathfrak{N}%
}^{\perp }T^{\ast }\Gamma _{m}P_{\mathfrak{N}}\geq 0,
\end{equation*}%
or, equivalently, if and only if
\begin{equation}
H-\overset{m-1}{\underset{j,k=0}{\sum }}b_{m+j+1}s_{jk}b_{m+k+1}\geq 0.
\label{stilt1}
\end{equation}%
Since
\begin{equation*}
Q:=\overset{m-1}{\underset{j,k=0}{\sum }}b_{m+j+1}s_{jk}b_{m+k+1}
\end{equation*}%
is positive, all numbers $H$ generating non-negative extensions $\widetilde{A%
}$ and hence the solutions of the Stieltjes problem, must be positive
definite and, moreover, satisfy the inequality $H\geq Q$. Notice that the
requirement $\widetilde{A}> 0$ excludes the equality in (\ref{stilt1}).

To express the inequality (\ref{cond2})interms of the given moments $%
b_{0},b_{1},b_{2},\ldots ,b_{2m}$ remind that the operator $A_{0}$ can be
represented as the operator of multiplication by the independent variable in
the space $\mathrm{L}_{m-1}$ of polynomials of degree $\leq m$ defined on
the subspace $\mathrm{L}_{m-1}$ of polynomials of degree $\leq m-1$. Let us
denote by $d_{k}(t),\,k=0,...,m,$ the set of orthogonal polynomials in $%
\mathrm{L}_{\sigma }^{2}$ with respect to any measure $d\sigma (t)$
satisfying (\ref{e1}),
\begin{equation}
d_{0}(t)=\frac{1}{\sqrt{b_{0}}},d_{k}(t)=\frac{1}{\sqrt{\Delta _{k}\Delta
_{k-1}}}\det \left\Vert
\begin{array}{cccc}
b_{0} & b_{1} & ... & b_{k} \\
b_{1} & b_{2} & ... & b_{k+1} \\
. & . & ... & . \\
b_{k-1} & b_{k} & ... & b_{2k-1} \\
1 & t & ... & t^{k}%
\end{array}%
\right\Vert ,\,\,k=1,2,...,m,  \label{orth}
\end{equation}%
where
\begin{equation}
\Delta _{k}=\det \left\Vert
\begin{array}{cccc}
b_{0} & b_{1} & ... & b_{k} \\
b_{1} & b_{2} & ... & b_{k+1} \\
. & . & ... & . \\
b_{k} & b_{k+1} & ... & b_{2k}%
\end{array}%
\right\Vert .  \label{determ}
\end{equation}%
The operator $A_{00}$ in $\mathrm{L}_{m-1}$ acts thereafter on arbitrary
polynomials $q\in \mathrm{L}_{m-1}$ as follows
\begin{equation}
\left( A_{00}q\right) (t)=tq(t)-\beta _{m-1}\langle q,d_{m-1}\rangle _{%
\mathrm{L}_{m-1}}d_{m}(t),\,\beta _{m-1}=\frac{\sqrt{\Delta _{m}\Delta _{m-2}%
}}{\Delta _{m-1}}.  \label{Azerozero}
\end{equation}%
Hence, for $q\in \mathrm{L}_{m-1}$ we have
\begin{equation}
\left( Gq\right) (t)=\left( [A_{0}-A_{00}]q\right) (t)=\beta _{m-1}\langle
q,d_{m-1}\rangle _{\mathrm{L}_{m-1}}d_{m}(t),  \label{A00add}
\end{equation}%
and
\begin{equation}
\left( \left[ A_{00}-zI_{m-1}\right] ^{-1}q\right) (t)=\frac{1}{t-z}\left[
q(t)-\frac{q(z)}{d_{m}(z)}d_{m}(t)\right] .  \label{resolv0}
\end{equation}%
For the scalar Hausdorff problem $\dim \mathcal{N}=\dim [\mathrm{L}%
_{m}\ominus \mathrm{L}_{m-1}]=1$ and $d_{m}(t)$ is a unit vector in $%
\mathcal{N}$. Therefore in the scalar case by (\ref{A00add}) and (\ref%
{resolv0}), and the condition (\ref{cond1}), any measure $d\sigma (t)$
satisfying (\ref{e1}) has the form
\begin{equation}
\begin{array}{c}
\beta _{m-1}^{2}\int\limits_{0 }^{\infty }\frac{1}{t}\left[ d_{m-1}(t)-%
\frac{d_{m-1}(0)}{d_{m}(0)}d_{m}(t)\right] d_{m-1}(t)d\sigma (t)-\lambda
\\
-\beta _{m-1}^{2}\int\limits_{0}^{\infty }\frac{1}{t-\lambda }\left[
d_{m-1}(t)-\frac{d_{m-1}(\lambda )}{d_{m}(\lambda )}d_{m}(t)\right]
d_{m-1}(t)d\sigma (t)\leq 0,\,\lambda \geq \Lambda .%
\end{array}
\label{cond3}
\end{equation}%
Notice further that for any $\lambda ,\mu $ in accordance to the
Christoffel-Darboux identity for orthogonal polynomials  \cite{Akh}
\begin{equation}
\begin{array}{c}
\beta _{m-1}\frac{d_{m-1}(\lambda )d_{m}(\mu )-d_{m-1}(\mu )d_{m}(\lambda )}{%
\mu -\lambda } \\
=\sum\limits_{k=0}^{m-1}d_{k}(\lambda )d_{k}(\mu ) \\
=-\frac{1}{\Delta _{m-1}}\det \left\Vert
\begin{array}{ccccc}
0 & 1 & \lambda  & ... & \lambda ^{m-1} \\
1 & b_{0} & b_{1} & ... & b_{m-1} \\
\mu  & b_{1} & b_{2} & ... & b_{m} \\
. & . & . & ... & . \\
\mu ^{m-1} & b_{m-1} & b_{m} & ... & b_{2m-2}%
\end{array}%
\right\Vert .
\end{array}
\label{ChD}
\end{equation}%
By (\ref{ChD}) one can
rewrite (\ref{cond3}) in the form
\begin{equation}
\beta _{m-1}\frac{d_{m-1}(\lambda )d_{m}(0)-d_{m-1}(0)d_{m}(\lambda )}{%
d_{m}(\lambda )d_{m}(0)}-\lambda \leq 0,\,\lambda \geq \Lambda .
\label{cond4}
\end{equation}%
By our assumptions the consecutive numbers $b_{1},...,b_{2m-1}$ can be
considered as the moments
\begin{equation*}
b_{0}^{1}=b_{1},...,b_{2m-2}^{1}=b_{2m-1}
\end{equation*}%
of a non-negative measure $td\sigma (t)$, where $d\sigma (t)$ is any
solution of the truncated Stieltjes moment for the given sequence $%
b_{0},...,b_{2m}$. Let us denote by $d_{k}^{1},\,$\ $k=0,...,m-1,$ the
system of orthogonal polynomials for the set of moments $%
b_{0}^{1},...,b_{2m-2}^{1}$ and by $\Delta _{k}^{1}$ the determinants $\Vert
b_{p+q}^{1}\Vert _{p,q=0}^{k}$. It follows from (\ref{cond4}) and (\ref{ChD}%
) that the condition (\ref{cond1}) can be represented in the equivalent form
\begin{equation}
\sqrt{\frac{\Delta _{m}\Delta _{m-2}^{1}}{\Delta _{m-1}\Delta _{m-1}^{1}}}%
\frac{d_{m-1}^{1}(\lambda )}{d_{m}(\lambda )}\leq 1.  \label{critHausd}
\end{equation}%
The last inequality permits to specify the condition d) in the solvability
criterion of the truncated Hausdorff problem.

\begin{thm}
\label{second} For the given system of moments $b_{0},...,b_{2m}$ satisfying
conditions a) - c) of Theorem \ref{first} there is at least one solution of
the truncated Stieltjes problem with non-negative measure concentrated on
the interval $[0,\Lambda ]$, i.e., there is a solution of the truncated
Hausdorff problem for the interval $[0,\Lambda ]$ if and only if the matrix $%
\Lambda \Gamma _{m-1}-\Gamma _{m-1}^{(1)}$ is non-negative and for any $%
\lambda \leq \Lambda $ the inequality (\ref{critHausd}) holds.

Under the above conditions the truncated Hausdorff problem has unique
solution if and only if
\begin{equation}
\sqrt{\frac{\Delta _{m}\Delta _{m-2}^{1}}{\Delta _{m-1}\Delta _{m-1}^{1}}}%
\frac{d_{m-1}^{1}(\Lambda )}{d_{m}(\Lambda )}=1.  \label{uniqHausd}
\end{equation}
\end{thm}

\section{Description of canonical solutions of the truncated Hausdorff
problem}

Let us assume that the conditions of Theorem \ref{second} hold. We denote by
$e_{r}(t),$ $r=1,...,m,$ the system of conjugate polynomials:
\begin{equation*}
e_{r}(t)=\int\limits_{-\infty }^{\infty }\frac{d_{r}(t)-d_{r}(t^{\prime })}{%
t-t^{\prime }}d\sigma (t^{\prime }),
\end{equation*}%
where $d\sigma (t)$ is any solution of the truncated Hamburger problem for
the set of moments $b_{0},...,b_{2m}$. All canonical solutions (that is
those generated by the self-adjoint extensions of the symmetric operator $%
A_{0}$ in $\mathrm{L}_{m}$) of the truncated Stieltjes moment problem are,
according to \cite{ATU}, described by the formula
\begin{equation}
\overset{\infty }{\underset{0}{\int }}\frac{d\sigma _{H}\left( t\right) }{t-z%
}=-\frac{e_{m}\left( z\right) (R_{H}+z)-e_{m-1}\left( z\right) }{d_{m}\left(
z\right) \left( R_{H}+z\right) -d_{m-1}\left( z\right) },  \label{stielt1}
\end{equation}%
\begin{equation}
R_{H}=\left( \Gamma _{m}^{-1}\right) _{mm}^{-1}\Lambda _{m}-H\left( \Gamma
_{m}^{-1}\right) _{mm}\ ,\quad \mathrm{Im}z>0,  \label{mom8b}
\end{equation}%
where
\begin{equation}
\Lambda _{m}=\left( \Gamma _{m}^{-1}\right) _{m-1,m}-b_{m+1}\left( \Gamma
_{m}^{-1}\right) _{mm}^{-2}  \label{stielt3}
\end{equation}%
and is $H$ the parameter such that
\begin{equation}
H=\tau +\overset{m-1}{\underset{j,k=0}{\sum }}b_{m+j+1}s_{jk}b_{m+k+1},
\label{stielt4}
\end{equation}%
where $\tau $ is any non-negative number. The application of the above
arguments to the Hausdorff problem yields

\begin{thm}
Among all canonical solutions of the truncated Stieltjes problem for a given
set of moments $b_{0},...,b_{2m}$ which satisfy the conditions a) - d) of
Theorem \ref{first}, those and only those are canonical solution of the
truncated Hausdorff problem for which the parameter $\tau $ in (\ref{stielt4}%
) satisfies the condition
\begin{equation}
0\leq \tau \leq \Lambda \left( 1-\sqrt{\frac{\Delta _{m}\Delta _{m-2}^{1}}{%
\Delta _{m-1}\Delta _{m-1}^{1}}}\frac{d_{m-1}^{1}(\Lambda )}{d_{m}(\Lambda )}%
\right) .  \label{hausd2}
\end{equation}
\end{thm}

\section{Truncated Hamburger problem with gap}

Having disposed of the problems related to the truncated Hausdorff problem,
one can turn now directly to the local moment problem for a given interval $%
[0,\Lambda ],$ i.e., the truncated Hamburger moment problem in which
along with the first $2n+1$ moments
\begin{equation}
a_{k}=\int\limits_{-\infty }^{\infty }t^{k}d\sigma (t),\,\,0\leq k\leq 2n,
\label{loc1}
\end{equation}%
of the sought measure $d\sigma (t),$ its $2m+1,\,n\leq m,$ local moments
\begin{equation}
b_{k}=\int\limits_{0}^{\Lambda }t^{k}d\sigma (t),\,\,0\leq k\leq 2m,
\label{loc2}
\end{equation}%
are given also.

A possible approach to the solution of the local moment problem consists in
the representation of the sought measure $d\sigma (t)$ as the sum
\begin{equation*}
d\sigma (t)=d\sigma _{\Lambda }(t)+d\sigma _{\Lambda }^{\perp }(t),
\end{equation*}%
where the measure $d\sigma _{\Lambda }(t)$ is concentrated on the segment $%
[0,\Lambda ]$, while the function $\sigma _{\Lambda }^{\perp }(t)$ has no
growth points on $[0,\Lambda ]$.

The retrieval of $d\sigma _{\Lambda }(t)$ is reduced to the above Hausdorff
problem on the interval $[0,\Lambda ]$ for the given set of moments $%
b_{0},...,b_{2m}$. The quest of $d\sigma _{\Lambda }^{\perp }(t)$ consists
in the search of some special solutions $d\widetilde{\sigma }(t)$ of the
truncated Hamburger moment problem, which satisfy the additional restriction
\begin{equation}
\widetilde{\sigma }(\Lambda -0)-\widetilde{\sigma }(+0)=0,  \label{gap}
\end{equation}%
for the set of moments
\begin{equation}
\begin{array}{c}
c_{k}=\int\limits_{0}^{\Lambda }t^{k}d\widetilde{\sigma }(t)= \\
\int\limits_{-\infty }^{\Lambda -0}t^{k}d\sigma _{\Lambda }^{\perp
}(t)+\int\limits_{\Lambda +0}^{\infty }t^{k}d\sigma _{\Lambda }^{\perp }(t)=
\\
\int\limits_{-\infty }^{\infty }t^{k}d[\sigma (t)-\sigma _{\Lambda
}(t)]=a_{k}-b_{k},\,\,k=0,...,2n.%
\end{array}
\label{loc3}
\end{equation}%
We call the latter moment problem the truncated Hamburger moment problem
with the gap $[0,\Lambda ]$. We see that the  local moment
problem formulated above is reduced to the truncated moment problem with the gap.

The proposed approach to the solution of this problem consists in the
selection among the solutions $d\widetilde{\sigma }(t)$ of the truncated
Hamburger moment problem for the set of moments $c_{0},...,c_{2n}$ of those
satisfying additional condition (\ref{gap}). In this way, we notice first
that the necessary condition of solvability of the Hamburger problem for the
given moments is positive definiteness of the Hankel matrix $\widetilde{%
\Gamma }_{n}=\left( c_{j+k}\right) _{j,k=0}^{n}$. We will assume further
that this condition holds.

Let $d\widetilde{\sigma }(t)$ be a solution of the problem with the gap.
Since $t(t-\Lambda )\geq 0$ on $\mathbb{E}\setminus \lbrack 0,\Lambda ]$
then for any polynomial
\begin{equation*}
P_{n-1}(t)=\xi _{0}+\xi _{1}\ t+\xi _{2}\ t^{2}+\ldots +\xi _{n-2}\ t^{n-2}
\end{equation*}%
we have
\begin{equation}
\begin{array}{c}
\int\limits_{-\infty }^{\infty }t(t-\Lambda )|P_{n-1}(t)|^{2}d\widetilde{%
\sigma }(t) \\ =\int\limits_{-\infty }^{\Lambda }t(t-\Lambda )|P_{n-1}(t)|^{2}d%
\widetilde{\sigma }(t)+\int\limits_{\Lambda }^{\infty }t(t-\Lambda
)|P_{n-1}(t)|^{2}d\widetilde{\sigma }(t) \\
=\underset{j,k=0}{\overset{n-2}{\sum }}\left[ c_{j+k+2}-\Lambda c_{j+k+1}%
\right] \xi _{k}\overline{\xi _{j}}\geq 0.%
\end{array}
\label{cond5}
\end{equation}%
Therefore the positive definiteness of the matrix
\begin{equation*}
\widetilde{\Gamma }_{n-2}^{(2)}-\Lambda \widetilde{\Gamma }%
_{n-2}^{(1)}=\left( c_{j+k+2}-\Lambda c_{j+k+1}\right) _{j,k=0}^{n-2}
\end{equation*}%
is an additional necessary condition for the solvability of the Hamburger
moment problem with the gap for a given moments $c_{0},...,c_{2n}$.

To find sufficient conditions of solvability and find a description of
canonical solutions of the gap problem one can as above look at this
problem from the point of view of the extension theory. In other words,
taking the set of moments $c_{0},...,c_{2n}$ one can consider the Hilbert
space $\mathrm{L}_{n}$ of polynomials $P(t)=\xi _{0}+\xi _{1}\ t+\xi _{2}\
t^{2}+\ldots +\xi _{n}\ t^{n}$ with the norm
\begin{equation*}
\Vert P\Vert =\sqrt{\underset{j,k=0}{\overset{n}{\sum }}c_{j+k+2}\xi _{k}%
\overline{\xi _{j}}}
\end{equation*}%
and the symmetric operator $\widetilde{A}_{0}$ in $\mathrm{L}_{n}$ defined
as the multiplication by $t$ operator on the subspace $\mathrm{L}%
_{n-1}\subset \mathrm{L}_{n}$ of polynomials of degree not exceeding $n-1$.
Remind that all solutions of the corresponding Hamburger problem are
generated by the self-adjoint extensions $\widetilde{A}$ of $\widetilde{A}%
_{0}$. Any self-adjoint extension $\widetilde{A}$ of $\widetilde{A}_{0}$ is
a special extension of $\widetilde{A}_{0}$ onto the defect subspace $%
\mathcal{N}_{0}=\mathrm{L}_{n}\ominus \mathrm{L}_{n-1}$. In our case $\dim
\mathcal{N}_{0}=1$ and $\mathcal{N}_{0}$ consists of polynomials, which are
collinear to the orthogonal polynomial $p_{n}(t)$, which is associated with
the set of moments $c_{0},...,c_{2n}$. It is easy to verify that $\widetilde{%
A}$ is uniquely defined by the formula
\begin{equation}\begin{array}{c}
(\widetilde{A}p_{k})(t)=\alpha _{k}p_{ k}(t)+\beta _{k-1}p_{k-1}(t)+\beta _{k}p_{k+1}(t),  \\
 \alpha _{k}=\left( \widetilde{A}_{0}p_{k},p_{k}\right)_{\mathrm{L}_{n}}, \; \beta
_{k}=\frac{\sqrt{\Delta _{k}\Delta _{k-2}}}{\Delta _{k-1}},\,\,\Delta
_{k}=\det (c_{j+k})_{j,k=0}^{k}, \;  k=0,...,n-1,\\
(\widetilde{A}p_{n})(t)=\alpha _{\tilde{A}}p_{n}(t)+\beta _{n-1}p_{n-1}(t), \label{param}
\end{array}
\end{equation}%
where $\alpha _{\tilde{A}}$ is a real parameter defining the the extension $%
\widetilde{A}$.

With no limitations on values of  real $\alpha _{\tilde{A}}$ the expressions (\ref{param}) define the all self-adjoint extensions of $\widetilde{A}_{0}$ in $\mathrm{L}_{n}$ and generate in this way the all canonical solutions $\widetilde{\sigma}(t)$ of the  truncated Hamburger moment problem. Remind that they  are describrd by the Nevanlinna formula
\begin{equation}\label{nevanlinna}
\overset{\infty }{\underset{-\infty }{\int }}\frac{d\widetilde{\sigma }%
_{\alpha _{\tilde{A}}}\left( t\right)}{t-z}=-\frac{q_{n}\left( z\right) (\alpha _{\tilde{A}}
+z)-q_{n-1}\left( z\right) }{p_{n}\left( z\right) \left( \alpha _{\tilde{A}} +z\right)
-p_{n-1}\left( z\right)},\,\,\mathrm{\mathrm{Im}}z\geq 0,
\end{equation}
where $q_{k}(z)$ are corresponding conjugate polynomials,
\[
q_{k}(z)=\int\limits_{-\infty}^{\infty}\frac{p_{k}(t)-p_{k}(z)}{t-z}d\widetilde{\sigma }%
_{\alpha _{\tilde{A}}}\left( t\right).
\]
It follows from the Nevanlinna formula (\ref{nevanlinna}) that {\it those and only those $\alpha _{\tilde{A}}$ give the sought canonical solutions with $\Lambda $-gap for which the polynomials}
\begin{equation}\label{quasipol1}
M_{\tilde{A}}(t)=(t-\alpha_{\tilde{A}})p_{n}(t)+\beta_{n-1}p_{n-1}(t)
\end{equation}
{\it have no zeros in} $(0,\Lambda)$.

Note that in some cases the last condition may not be satisfied for any real $\alpha _{\tilde{A}}$, that is the Hamburger moment problem with given gap may be not solvable while the corresponding problem without the gap demand may have infinitely many solutions. Indeed, remember that for any real $\alpha _{\tilde{A}}\neq 0$ the zeros of polynomial $M_{\tilde{A}}(t)$ are real and simple and between any two zeros of $p_{n-1}(t)$ there is at least one zero of $M_{\tilde{A}}(t)$. Hence, if $p_{n-1}(t)$ has two or more zeros in $(0,\Lambda)$ the corresponding Hamburger problem with this gap has no solutions.

 To get the solvability condition of the truncated Hamburger proble with gap observe that any self-adjoint extension $\widetilde{A}$ of $\widetilde{A}_{0}
$ in $\mathrm{L}_{n}$ generates a self-adjoint extension $Q_{\widetilde{A}}$
of the symmetric operator $Q_{0}$, which is defined on the subspace $\mathrm{%
L}_{n-2}\subset \mathrm{L}_{n}$ and acts as the operator of multiplication
by the polynomial $t(t-\Lambda )$, $Q_{\widetilde{A}}$ is simply $\widetilde{%
A}(\widetilde{A}-\Lambda \cdot I)$, where $I$ is the unity operator in $%
\mathrm{L}_{n}$. At the same time $\widetilde{A}$ generates a gap extension
if and only if $\widetilde{A}$ has no eigenvalues on the segment $[0,\Lambda
]$, that is if and only if $Q_{\widetilde{A}}$ is a positive operator.

Let $\mathcal{N}_{1}$ denote the subspace $\mathrm{L}_{n}\ominus \mathrm{L}%
_{n-2}$. With respect to the representation of $\mathrm{L}_{n}=\mathrm{L}%
_{n-2}\oplus \mathcal{N}_{1}$, write an extension $Q_{\widetilde{A}}$ in the
block form
\begin{equation}
Q_{\widetilde{A}}=\left(
\begin{array}{cc}
Q_{00} & K \\
K^{\ast } & W_{\widetilde{A}}%
\end{array}%
\right) .  \label{block1}
\end{equation}%
Remind that the $(n-2)\times (n-2)$ block $Q_{00}$ of $Q_{\widetilde{A}}$
does not depend on the choice of the extension $\widetilde{A}$ and by (\ref%
{cond5}) it is a positive operator (positive definite matrix). Using the
Schur-Frobenius factorization
\begin{equation}
Q_{\widetilde{A}}=\left(
\begin{array}{cc}
I & 0 \\
K^{\ast }Q_{00}^{-1} & I%
\end{array}%
\right) \left(
\begin{array}{cc}
Q_{00} & 0 \\
0 & W_{\widetilde{A}}-K^{\ast }Q_{00}^{-1}K%
\end{array}%
\right) \left(
\begin{array}{cc}
I & Q_{00}^{-1}K \\
0 & I%
\end{array}%
\right) ,  \label{Schur3}
\end{equation}%
we see that under our assumptions the operator $Q_{\widetilde{A}}$ is
positive if and only if the operator ($2\times 2$ matrix) $W_{\widetilde{A}%
}-K^{\ast }Q_{00}^{-1}K$ is positive (positive definite).

To represent the positivity condition for the $2\times 2$ matrix $W_{\widetilde{A}%
}-K^{\ast }Q_{00}^{-1}K$  in an explicit form let us assume that operator $Q_{\widetilde{A}}$ is invertible and  write  $Q_{\widetilde{A}}^{-1}$ with respect to the splitting $\mathrm{L}_{n}=\mathrm{L}_{n-2}\oplus \mathcal{N}_{1}$ in the block form
\begin{equation}
Q_{\widetilde{A}}^{-1}=\left(
\begin{array}{cc}
Y_{00} & X \\
X^{\ast } & Z_{\widetilde{A}}%
\end{array}%
\right) .  \label{block2}
\end{equation}%
 Note that
\begin{equation}\label{unexp}
 Z_{\widetilde{A}}^{-1}=W_{\widetilde{A}}-K^{\ast }Q_{00}^{-1}K.
\end{equation}
Indeed, for any invertible block-matrix
\[
\mathbf{L}=\left(
\begin{array}{cc}
A & C \\
C^{\ast } & B%
\end{array}%
\right)
\]
with invertible diagonal block $B$ a direct calculatiion shows that
\begin{equation}\label{standard}
\left(
\begin{array}{cc}
0 & 0 \\
0 & B^{-1}%
\end{array}%
\right) =\mathbf{L}^{-1}- \mathbf{L}^{-1}\left(
\begin{array}{cc}
G_{A}^{-1} & 0 \\
0 & 0%
\end{array}%
\right)\mathbf{L}^{-1},
\end{equation}
where $G_{A}$ is the left upper block of  $\mathbf{L}^{-1}$. Replacing in (\ref{standard}) $\mathbf{L}^{-1}$ by $\mathbf{L}$ and applying the obtained relation to $Q_{\tilde{A}}$ yields (\ref{unexp}). Hence  $W_{\widetilde,e{A}}-K^{\ast }Q_{00}^{-1}K$ is positive definite if and only if the corresponding inverse matrix $Z_{\widetilde{ A}}$ is positive definite.

To obtain a condition that guarantees that $Z_{\widetilde{ A}}>0$ in terms of given moments $c_{0},...,c_{2n}$ let us write down the matrix $\hat{Z}_{\widetilde{ A}}$  of $Z_{\widetilde{ A}}$ for the basis of $e_{1}, e_{2}$ in $ \mathcal{N}_{1}$ of the orthogonal polynomials $p_{n-1}(t), p_{n}(t)$, respectively. To this end remember that for any $z$ which is not an eigenvalue of $\tilde{A}$ the resolvent $ \left(\tilde{A}-zI\right)^{-1}$ acts on any polynomial $g\in \mathrm{L}_{n}$  by formula
\begin{equation} \label{resolv}
\left[ \left(\tilde{A}-zI\right)^{-1}g\right](t)=\frac{g(t)M_{\tilde{A}}(z)-g(z)M_{\tilde{A}}(t)}{M_{\tilde{A}}(z)}\frac{1}{t-z}
\end{equation}
with introduced as in (\ref{quasipol1})  polynomial $M_{\tilde{A}}(t)$.
Applying, in particular,  (\ref{resolv}) to $p_{n-1}(t)$ and $p_{n}(t)$ with account of  the Christoffel-Darboux identity for the polynomials $\{ p_{k}(t)\}_{0}^{n}$ we  obtain:
\begin{equation}\label{resolv1a}
\left[ \left(\tilde{A}-zI\right)^{-1}p_{n-1}\right](t)=\frac{1}{M_{\tilde{A}}(z)}\left[ \frac{z-\alpha _{\tilde{A}}}{\beta_{n-1}}\sum\limits_{k=0}^{n-1}p_{k}(z)p_{k}(t) \, +p_{n-1}(z)p_{n}(t)\right],
\end{equation}
\begin{equation}\label{resolv1b}
\left[ \left(\tilde{A}-zI\right)^{-1}p_{n}\right](t)=\frac{1}{M_{\tilde{A}}(z)}\sum\limits_{k=0}^{n}p_{k}(z)p_{k}(t).
\end{equation}
As follows, the right lower $2\times 2$-block $R_{\tilde{A}}(z)$  of the resolvent $ \left(\tilde{A}-zI\right)^{-1}$ for the basis $\{p_{k}\}_{0}^{n}$ has form
\begin{equation}\label{resolv3}
R_{\tilde{A}}(z)=\frac{1}{M_{\tilde{A}}(z)}\left( \begin{array}{cc} \frac{z-\alpha _{\tilde{A}}}{\beta_{n-1}}p_{n-1}(z),  & p_{n-1}(z) \\ p_{n-1}(z), &  p_{n}(z) \end{array} \right).
\end{equation}
Since
\[
Q_{\widetilde{A}}=\frac{1}{\Lambda}\left[(\widetilde{A}-\Lambda \cdot I)^{-1} -\widetilde{A}^{-1}\right],
\]
then by  (\ref{resolv3} )
\begin{equation}\label{resolv4}\begin{array}{c}
Z_{\widetilde{ A}}=\frac{1}{M_{\tilde{A}}(\Lambda)M_{\tilde{A}}(0)} \\  \times \left(\begin{array}{cc} \frac{\Lambda - \alpha _{\tilde{A}}}{\beta_{n-1}} p_{n-1}(\Lambda)M_{\tilde{A}}(0)+\frac{\alpha _{\tilde{A}}}{\beta_{n-1}}p_{n-1}(0)M_{\tilde{A}}(\Lambda), &  p_{n-1}(\Lambda)M_{\tilde{A}}(0)- p_{n-1}(0)M_{\tilde{A}}(\Lambda) \\  p_{n-1}(\Lambda)M_{\tilde{A}}(0)- p_{n-1}(0)M_{\tilde{A}}(\Lambda), &  p_{n}(\Lambda)M_{\tilde{A}}(0)- p_{n}(0)M_{\tilde{A}}(\Lambda) \end{array} \right).
\end {array}
\end{equation}
Setting $$ h_{s}(\lambda , \mu) = \sum\limits_{k=0}^{s}p_{k}(\lambda)p_{k}(\mu)$$ and applying the Christoffel-Darboux identity we deduce from the representation (\ref{resolv4}) that $Z_{\widetilde{ A}}$ is positive definite if and only if the following inequalities hold:
\begin{equation}\label{firsta}
\frac{h_{n}(\Lambda ,0)}{M_{\tilde{A}}(\Lambda)M_{\tilde{A}}(0)}>0,
\end{equation}
\begin{equation}\label{secondb}
\begin{array}{c}
W(\alpha _{\tilde{A}})=\frac{h_{n-1}(\Lambda ,0)}{\beta_{n-1}^{2} M_{\tilde{A}}(\Lambda)M_{\tilde{A}}(0)}\left\{ - \left[h_{n}(\Lambda ,0)+h_{n-1}(\Lambda ,0) \right]\alpha _{\tilde{A}}^{2}\right. \\ +\left[\Lambda  h_{n}(\Lambda ,0)h_{n-1}(\Lambda ,0)+2\beta_{n-1}p_{n}(\Lambda)p_{n-1}(0)h_{n-1}(\Lambda ,0)\right]\alpha _{\tilde{A}} \\ \left. +\beta_{n-1}p_{n-1}(0)\left[ \beta_{n-1}p_{n-1}(\Lambda)-\Lambda p_{n}(\Lambda)\right]\right\}>0.
\end{array}
\end{equation}

Summarizing the assertions obtained above yields the fillowing
\begin{thm}\label{gapsolv}
For the given set of moments $c_{0},...,c_{2n}$ the Hamburger moment problem
with the gap $[0,\Lambda ]$ is solvable and has infinitely many solutions if
and only if

\begin{itemize}
  \item {the Hankel matrices $\widetilde{\Gamma }_{n}=\left(
c_{j+k}\right) _{j,k=0}^{n}$ and
\begin{equation*}
\widetilde{\Gamma }_{n-2}^{(2)}-\Lambda \widetilde{\Gamma }%
_{n-2}^{(1)}=\left( c_{j+k+2}-\Lambda c_{j+k+1}\right) _{j,k=0}^{n-2}
\end{equation*}
are positive definite;}
  \item {the inequality (\ref{firsta}) holds;}
  \item {the roots of quadratic trinomial $W(\alpha _{\tilde{A }})$ are real and different.}
\end{itemize}

Under the above conditions the set of canonical solutions $d\widetilde{%
\sigma }_{\tau }(t)$ of the Hamburger problem with the gap $[0,\Lambda ]$ is
described by the Nevanlinna formula (\ref{nevanlinna}), where $\alpha _{\tilde{A }}$ runs the segment of real axis where $W(\alpha)\geq 0$.
\end{thm}

\end{document}